\documentclass{amsart}

\usepackage{amsmath,amssymb,amsrefs}
\usepackage[colorlinks]{hyperref}

\newcommand{\Z}{\mathbb Z}
\newcommand{\C}{\mathcal C}
\newcommand{\f}{\mathfrak f}

\newtheorem{lemma}{Lemma}[section]
\newtheorem{proposition}[lemma]{Proposition}
\newtheorem{corollary}[lemma]{Corollary}
\newtheorem{theorem}[lemma]{Theorem}
\theoremstyle{remark}
\newtheorem{remark}[lemma]{Remark}
\theoremstyle{definition}
\newtheorem{definition}[lemma]{Definition}

\begin{document}
\title[An inductive proof of the coin problem]{An inductive proof of the Frobenius coin problem of two denominations}
\author[G.~Kapetanakis]{Giorgos Kapetanakis}
\address{Department of Mathematics, University of Thessaly, 3rd km Old National Road Lamia-Athens, 35100, Lamia, Greece}
\email{kapetanakis@uth.gr}
\author[I.~Rizos]{Ioannis Rizos}
\address{Department of Mathematics, University of Thessaly, 3rd km Old National Road Lamia-Athens, 35100, Lamia, Greece}
\email{ioarizos@uth.gr}
\date{\today}

\begin{abstract}
Let $a,b$ be positive, relatively prime, integers. We prove, using induction, that for every $d > ab-a-b$ there exist $x,y\in\Z_{\geq 0}$, such that $d=ax+by$. As a byproduct, we obtain a constructive recursive algorithm for identifying appropriate $x,y$ as above.
\end{abstract}



\maketitle
\section{Introduction}
Let $a,b$ be positive, relatively prime, integers. The well-known \emph{Frobenius coin problem of two denominations}, or just the \emph{coin problem}, is identifying the largest integer (if it exists) that is not in the set
\[ \C_{a,b} = \{ d\in\Z_{\geq 0} : d = ax+by\text{ for some } x,y\in\Z_{\geq 0} \} . \]
This number is called the \emph{Frobenius number} and we denote it by $\f_{a,b}$.
In fact, the Frobenius coin problem of two denominations was originally answered in 1882 by Sylvester.
\begin{theorem}[Sylvester] \label{main}
Let $a,b\in\Z_{>0}$ with $\gcd(a,b)=1$. Then $\f_{a,b} = ab-a-b$.
\end{theorem}
In the literature there are multiple proofs of this result, using elementary number theory, geometry, analytic methods, or even by exploring the semigroup structure of $\C_{a,b}$ and its connections with semigroup polynomials and cyclotomic polynomials. We refer the interested reader in \cites{alfonsin05,kapetanakisrizos23a,moore14} and the references therein.

In addition, for given ---usually relatively small--- $a$ and $b$, Theorem~\ref{main} is often proven, see \cite{epp10}*{Proposition~5.2.1}, or given as an exercise, see \cite{rosen12}*{Exercise~5(a), p.~342}, using induction. However, a purely inductive proof, for arbitrary $a$ and $b$ is not straightforward and is apparently missing from the existing literature. 
In this note, we fill this gap by providing one such elementary proof in Section~\ref{sec:main_proof} and, as a byproduct, in Section~\ref{conclusion}, we obtain a recursive algorithmic method for finding appropriate factors $x,y\in\Z_{\geq 0}$, such that $m=ax+by$, for every $m > \f_{a,b}+1$.
%
\section{Preliminaries}
It is clear that we will be working with the linear Diophantine equation
\begin{equation}\label{dioph}
d=ax+by,
\end{equation}
where, we assume that $0<a,b$ and $\gcd(a,b)=1$. We shall also assume that $d\geq 0$, since we are interested in finding positive integer solutions of Eq.~\eqref{dioph} and it is clear that the case $d<0$ is impossible. The following celebrated result from Number Theory is well-known.
\begin{theorem}[B\'ezout's identity] \label{bezout}
  Let $a,b\in\Z$ with $c=\gcd(a,b)$. There exist some $x',y'\in\Z$, such that $ax' +by'  = c$.
\end{theorem}
Clearly, Theorem~\ref{bezout} implies that Eq.~\eqref{dioph} is solvable. In particular, the following is true:
\begin{proposition} \label{gen_sol}
 Let $a,b\in\Z$ not both zero, with $\gcd(a,b)=1$. If $x_0,y_0\in\Z$ are such that $ax_0 +by_0  = d$, then the set of solutions of Eq.~\eqref{dioph} is the (infinite) set
\[ S = \{ (x_0-kb,y_0+ka) : k\in\Z\} . \]
\end{proposition}
However, it is not guaranteed that a solution $(x,y)\in S$ of Eq.~\eqref{dioph} is such that $x,y\geq 0$. For this reason we will call the expression $d=ax+by$ \emph{acceptable} if $x,y\geq 0$ and \emph{unacceptable} otherwise.
%
\begin{corollary}\label{xy}
Let $a,b\in\Z_{>1}$ with $\gcd(a,b)=1$. There exist integers $x_1,x_2,y_1,y_2$ with the following properties:
\begin{enumerate}
  \item $ax_1+by_1 = 1$, \label{ax1+by1=1}
  \item $ax_2+by_2=1$, \label{ax2+by2=1}
  \item $0<x_1<b$, $-a<y_1<0$, $-b<x_2<0$, $0<y_2<a$, \label{xy_signs}
  \item $x_1-x_2=b$, \label{x1-x2=b}
  \item $y_2-y_1=a$. \label{y2-y1=a}
\end{enumerate}
\end{corollary}
\begin{proof}
Immediate from Proposition~\ref{gen_sol} and the well-ordering principle.
\end{proof}
\begin{definition} \label{minimal}
We will call the expressions $ax_1+by_1=1$ and $ax_2+by_2=1$ from Corollary~\ref{xy} \emph{minimal unit expressions}.
\end{definition}
\begin{remark}
It is clear that, if $a,b> 1$ and $ax+by=1$, then one of $x,y$ has to be positive and the other negative,
while if the absolute value of one these numbers increases, it forces the absolute value of the other one to increase as well. Additionally, Proposition~\ref{gen_sol} implies that, in fact, both combinations of signs, i.e., $x>0$ and $y<0$ and vice versa, do always exist and both appear in infinite occasions.
In other words, Definition~\ref{minimal} makes sense and both minimal unit expressions are well-defined, unique and exist for every appropriate choice of $a$ and $b$.
\end{remark}
%
\section{The main proof} \label{sec:main_proof}
First, fix some $a,b\in\Z_{>1}$ that are relatively prime. We omit the cases where $a=1$ or $b=1$ as these cases are trivial. Furthermore, from now on, we fix the numbers $x_1,y_1,x_2,y_2$ as the ones defined in Corollary~\ref{xy}.
%
%
%
We will first show that $ab-a-b\not\in\C_{a,b}$.
\begin{lemma}\label{ab-a-b}
Let $a,b\in\Z_{>1}$ with $\gcd(a,b)=1$. Then $ab-a-b\not\in\C_{a,b}$.
\end{lemma}
\begin{proof}
Assume that $ab-a-b\in\C_{a,b}$. Then there exist $x,y\in\Z_{\geq 0}$, such that
\[ ab-a-b = ax+by . \]
It follows that $a(b-1-x) = b(y+1)$, hence $b-1-x>0$ and, since $\gcd(a,b)=1$, $b\mid b-1-x \Rightarrow b\leq b-1-x$, a contradiction.
\end{proof}
Then, we show that $ab-a-b+i\in\C_{a,c}$, for $i=1,2$.
\begin{lemma}\label{ab-a-b+1,2}
Let $a,b\in\Z_{>1}$ with $\gcd(a,b)=1$. Then $ab-a-b+i\in\C_{a,b}$, for $i=1,2$ and, given an acceptable expression of $ab-a-b+1$, we obtain an acceptable expression of $ab-a-b+2$ by adding a minimal unit expression to it. 
\end{lemma}
\begin{proof}
It is trivial to check, given item~\eqref{xy_signs} of Corollary~\ref{xy}, that
\begin{equation}\label{ab-a-b+1} ab-a-b+1 = a(b-1+x_2) + b(y_2-1) \end{equation}
is an acceptable expression of $ab-a-b+1$, hence $ab-a-b+1\in\C_{a,b}$. As an exercise, the reader can prove that the above is actually the unique acceptable expression of $ab-a-b+1$.

Next, items~\eqref{ax1+by1=1} and \eqref{ax2+by2=1} of Corollary~\ref{xy} yield
\[ ab-a-b+2 = a(b-1+2x_2) + b(2y_2-1)= a(b-1+x_1+x_2) + b(y_1+y_2-1) \]
and we will show that one of these expressions is acceptable. We have already seen that $b-1+x_2\geq 0$ and $y_2-1\geq 0$, and from item~\eqref{xy_signs} of Corollary~\ref{xy}, we have that $x_1,y_2>0$. Thus, it suffices to show that
\[
b-1+2x_2\geq 0 \ \text{ or } \ y_1+y_2-1 \geq 0 .
\]
We assume that $b-1+2x_2< 0$ and $y_1+y_2-1 < 0$. From item~\eqref{x1-x2=b} of Corollary~\ref{xy}, we obtain:
\[  b-1+2x_2 < 0 \Rightarrow x_1+x_2 \leq 0 \Rightarrow ax_1+ax_2 \leq 0 . \]
We also have that
\[ y_1+y_2-1 < 0 \Rightarrow y_1+y_2 \leq 0 \Rightarrow by_1+by_2 \leq 0 . \]
We combine the above with items~\eqref{ax1+by1=1} and \eqref{ax2+by2=1} of Corollary~\ref{xy} and get that $2\leq 0$, a contradiction.
\end{proof}
%
%
Finally, we prove the following:
\begin{proposition}\label{m}
Let $a,b\in\Z_{>1}$, with $\gcd(a,b)=1$, and $d\geq ab-a-b+2$. Then $d-1\in\C_{a,b}$ and, if $d-1=a(x_0+x_i)+b(y_0+y_i)$ is an acceptable expression of $d-1$ for $i=1$ or $i=2$, then one of the expressions
\[ d = a(x_0+x_i+x_1)+b(y_0+y_i+y_1) = a(x_0+x_i+x_2)+b(y_0+y_i+y_2) \]
is acceptable, i.e., $d\in\C_{a,b}$.
\end{proposition}
\begin{proof}
We will use induction on $d$. Clearly, Lemma~\ref{ab-a-b+1,2} implies the desired result for $d=ab-a-b+2$.

Suppose the desired results holds for some $k>ab-a-b+1$. This implies that, without loss of generality, we may assume that
\begin{equation}\label{k} k=a(x_0+x_1) + b(y_0+y_1) \end{equation}
is an acceptable expression of $k$, that is,
\begin{equation}\label{k1} x_0+x_1 \geq 0 \text{ and } y_0+y_1 \geq 0 , \end{equation}
while, clearly,
\begin{equation}\label{k-1} k-1 = ax_0+by_0 . \end{equation}

From Eq.~\eqref{k} and items~\eqref{ax1+by1=1} and \eqref{ax2+by2=1} of Corollary~\ref{xy}, we get that
\[ k+1 = a(x_0+2x_1) + b(y_0+2y_1) = a(x_0+x_1+x_2) + b(y_0+y_1+y_2) , \]
where it suffices to show that at least one of the above expressions is acceptable. Further, from Eq.~\eqref{k1} and item~\eqref{xy_signs} of Corollary~\ref{xy}, it suffices to show that
\[ y_0+2y_1\geq 0 \text{ or } x_0+x_1+x_2\geq 0 . \]
Assume that $y_0+2y_1< 0$ and $x_0+x_1+x_2< 0$. Item~\eqref{x1-x2=b} of Corollary~\ref{xy} yields that
\[ x_0 + x_1 + x_2 <0 \Rightarrow 2x_1 \leq b-1-x_0 \Rightarrow 2ax_1 \leq ab - a - ax_0 . \]
Likewise, we also have that
\[ y_0 + 2y_1 <0 \Rightarrow 2y_1 \leq -1-y_0 \Rightarrow 2by_1\leq -b-by_0 . \]
We combine the above and get that
\[ 2(ax_1+by_1) \leq ab-a-b-(ax_0+by_0) . \]
The above, combined with Eq.~\eqref{k-1} and item~\eqref{ax1+by1=1} of Corollary~\ref{xy}, yields
\[ 2 \leq ab-a-b-(k-1) \Rightarrow k\leq ab-a-b-1 , \]
a contradiction. The desired result follows.
\end{proof}
Theorem~\ref{main} follows directly from Lemmas~\ref{ab-a-b} and \ref{ab-a-b+1,2} and Proposition~\ref{m}.
\section{Concluding remarks}\label{conclusion}
In this work, we provided an inductive proof of Theorem~\ref{main} that was missing from the existing literature. As already mentioned, it is sometimes even given as a simple example or an exercise to prove Theorem~\ref{main} for specific ---usually small--- values of $a$ and $b$. The approach in these cases relies on direct computations that lack the generality of our proof and does not provide any insight on possible general proofs. In particular, the core idea of this, specific, approach can be summarized as follows:
\begin{enumerate}
  \item Show that $ab-a-b\not\in\C_{a,b}$.
  \item Show, with explicit examples, that $d_i := ab-a-b+i\in\C_{a,b}$ for $1\leq i\leq a$.\label{explicit}
  \item If $d \geq ab-a-b$, then, necessarily, $d = d_i + ka$, for some $1\leq i\leq a$ and $k\in\Z_{\geq 0}$. Given the acceptable expression of $d_i$ from the previous step, the above can easily give rise to an acceptable expression of $d$, thus $d\in\C_{a,b}$.
\end{enumerate}
Note that step~\eqref{explicit}, replies on brute force and explicit computations.
In addition to the general nature of our proof, a recursive method for finding acceptable expressions for any $d\geq \f_{a,b}$ is obtained from it. In particular, one can follow the steps below: 
\begin{enumerate}
  \item Find the minimal unit expressions.
  \item Compute the acceptable expression for $ab-a-b+1$, as given in Eq.~\eqref{ab-a-b+1}.
  \item An acceptable expression for any $d>ab-a-b+1$ is obtained by adding one of the minimal unit expressions on the corresponding acceptable expression of $d-1$.
\end{enumerate}
\begin{remark}
Given an acceptable expression of $d-1$, we showed that by adding the minimal unit expressions to it we will get at least one acceptable expression of $d$. However, it is possible for both minimal unit expressions to give rise to acceptable expressions of $d$. In fact, all the acceptable expressions of all $d\geq ab-a-b+1$ can be given as a tree with its root being the (unique) acceptable expression of $ab-a-b+1$ and each further branch constructed with the above method. We leave this proof to the interested reader.
\end{remark}
\bibliography{references}
%
%
%
%
%
%
%
\end{document}